\pdfoutput=1
\documentclass[11pt]{amsart}

\usepackage{epsfig,overpic,comment}%graphics,amssymb
\usepackage[usenames,dvipsnames,svgnames,table]{xcolor}
\usepackage[pagebackref,linktocpage=true,colorlinks=true,linkcolor=Blue,citecolor=BrickRed,urlcolor=RoyalBlue]{hyperref}
\usepackage[alphabetic,msc-links,abbrev]{amsrefs} %shortalphabetic,msc-links,backrefs
\usepackage[colorinlistoftodos,prependcaption,textsize=tiny, textwidth=0.75in]{todonotes}
\usepackage[a4paper, left=1.2in,right=1.2in,top=1.4in,bottom=1.7in]{geometry} %DukeJ format
\usepackage{esint}

\newtheorem{thm}{Theorem}[section]
\newtheorem{lem}[thm]{Lemma}

\newtheorem{remark}[thm]{Remark}

\newcommand{\R}{\mathbb{R}}

\newcommand\na{\nabla}
\newcommand\p{\partial}

\renewcommand{\S}{\mathcal S}

\numberwithin{equation}{section} % numbering with sections

\usepackage{tikz}
\usetikzlibrary{arrows,shapes,positioning}
\usetikzlibrary{decorations.markings}
\tikzstyle arrowstyle=[scale=1]
\tikzstyle directed=[postaction={decorate,decoration={markings,
    mark=at position .65 with {\arrow[arrowstyle]{stealth}}}}]
\tikzstyle reverse directed=[postaction={decorate,decoration={markings,
    mark=at position .65 with {\arrowreversed[arrowstyle]{stealth};}}}]
\usetikzlibrary{patterns}

\title[ ]
{A remark on an obstacle problem with lower regularity}

\author{Aram  Karakhanyan}
\address{School of Mathematics, University of Edinburgh, Peter Guthrie Tait Road
Edinburgh, 
EH9 3FD}
\email{aram6k@gmail.com}
\urladdr{www.maths.ed.ac.uk/$\sim$aram}

\date{\today \,(Last Typeset)}
\begin{document}

\begin{abstract}
We construct a monotone quantity for the classical obstacle problem 
with non-smooth obstacle,  and show that the blow-ups are homogeneous functions of degree $\alpha<2$.  
\end{abstract}

\maketitle

%----------------------
%     Section
%----------------------
\section{Introduction}\label{sec1}

Let $\Omega\subset \R^n$ be a bounded smooth domain and let 
\begin{equation}
K=\{ w : w-g\in W^{1, 2}_0(\Omega), w\ge \phi\}
\end{equation}
where $g\in W^{1, 2}(\Omega)$ is given Dirichlet data and $\phi\in W^{2, p}(\Omega)\cap C^{1, \beta}(\Omega)$ is the obstacle. We assume that $g\ge \phi$ on 
$\p \Omega$. 
The obstacle problem then can be formulated as follows: find a $u\in K$ such that 
\begin{eqnarray*}
\int _{\Omega }\left| \nabla u\right| ^{2}\leq \int _{\Omega }\left| \nabla w\right| ^{2}, \quad \forall w\in K.
\end{eqnarray*}
This problem has been extensively studied for smooth $\phi$, see \cite {MR454350}, \cite{Friedman}, \cite{MR567780},\cite{MR3648978}, \cite{MR3904453}, \cite{MR3855748}, \cite{MR1714335} and references therein.
Our aim is to study this problem for  $C^{1, \beta}$ regular $\phi$. 

\begin{thm}\label{thm-1}
Suppose that  $\phi\in C^{1, \beta}, \phi(0)=|\na \phi(0)|=0$ , $0\in \Omega$  and the following inequality holds
\begin{equation}
\left( \dfrac {|x|}{R}\right) ^{\alpha }\phi \left( R\dfrac {x}{|x|}\right) \geq \phi \left( x\right). 
\end{equation}
Then for every sequence $R_i\downarrow 0$ there is a subsequence $R_{i_m}$
such that $u_{m}=\frac{u(R_{i_m})}{R^{\alpha}_{i_m}}$ converges to some 
$u_0$ and $u_0$ is either homogeneous function of degree $\alpha:=1+\beta$ or
$u_0\equiv 0$. 
\end{thm}

\begin{remark}\label{rem-1}
Recall that $\na u$ and $\na \phi$ have  comparable moduli  of continuity, see page 46 \cite{Friedman}.  
\end{remark}

\begin{lem}\label{lem-1}
Let $u$ be as in Theorem \ref{thm-1}. Let 
\begin{eqnarray*}
A(R, u)=\frac1{R^{n+2(\alpha-1)}}\int_{B_R}|\na u|^2-\alpha \int_{\S}\left(\frac{u}{R^\alpha}\right)^2
\end{eqnarray*}
Then $A(R, u)$ is nondecreasing function and $A"(R)=0$ if and only if $u$ is a homogenous function of degree $\alpha$.
\end{lem}

\begin{proof}
Define 
\begin{eqnarray*}
w\left( x\right) :=\left( \dfrac {|x|}{R}\right) ^{\alpha }u\left( R\dfrac {x}{|x|}\right) \geq \left( \dfrac {|x|}{R}\right) ^{\alpha }\phi \left( R\dfrac {x}{|x|}\right) \geq \phi \left( x\right),  
\end{eqnarray*}
and
denote $y=R\dfrac {x}{|x|}$.
Then we have 
\begin{eqnarray*}
\nabla _{i}w\left( x\right) 
=
\alpha \left( \dfrac {|x|}{R}\right) ^{\alpha -1}\dfrac {x_{i}}{|x|R}u\left( R\dfrac {x}{|x|}\right) 
+
\left( \dfrac {|x|}{R}\right) ^{\alpha }u_{y_{m}}R\left( \dfrac {\delta _{ij}}{|x|}-\dfrac {x_{i}x_{m}}{|x|^{3}}\right). 
\end{eqnarray*}
Observe that 
\begin{eqnarray*}
\sum ^{n}_{i=1}x_{i}u_{y_i}-\sum ^{n}_{m=1}u_{y_m}x_{m}\sum ^{n}_{i=1}\dfrac {x^{2}_{i}}{|x|^{2}}=0.
\end{eqnarray*}
Consequently, 
\begin{eqnarray*}
\left| \nabla w\right| ^{2}
&=&
\alpha ^{2}\left| \dfrac {x}{R}\right| ^{2\left( \alpha -1\right) }\dfrac {1}{R^{2}}u^{2}\left( R\dfrac {x}{x}\right)
+
\left( \dfrac {|x|}{R}\right) ^{2\alpha }\dfrac {R^{2}}{|x|^{2}}\left[ \nabla u-\dfrac {x}{|x|}\left( \nabla u\dfrac {x}{|x|}\right) \right] ^{2}\\
&=&
\left( \dfrac {|x|}{R}\right) ^{2\left( \alpha -1\right) }
\left[ \alpha^2 \dfrac {u^{2}}{R^{2}}+\left| \na u\right| ^{2}-\left( \na u\dfrac {x}{|x|}\right) ^{2}\right]. 
\end{eqnarray*}

Next we compute the Dirichlet energy 
\begin{eqnarray*}
\int_{B_R} \left( \dfrac {|x|}{R}\right) ^{2\left( \alpha -1\right) }
\left[ \alpha^2 \dfrac {u^{2}}{R^{2}}+\left| \na u\right| ^{2}-\left( \na u\dfrac {x}{|x|}\right) ^{2}\right] \\
=
\frac1{R^{2(\alpha-1)}}\int_{\S} 
\left[ \alpha^2 \dfrac {u^{2}(R\theta)}{R^{2}}+\left| \na u(R\theta)\right| ^{2}-\left( \na u(R\theta)\theta\right) ^{2}\right] \int_0^R\rho^{2(\alpha-1)+n-1}d\rho\\
=
\frac{R^{n+2(\alpha-1)}}{n+2(\alpha-1)}\frac1{R^{2(\alpha-1)}}\int_{\S} 
\left[ \alpha^2 \dfrac {u^{2}(R\theta)}{R^{2}}+\left| \na u(R\theta)\right| ^{2}-\left( \na u(R\theta)\theta\right) ^{2}\right]\\
=
\frac{R^{n}}{n+2(\alpha-1)}\int_{\S} 
\left[ \alpha^2 \dfrac {u^{2}(R\theta)}{R^{2}}+\left| \na u(R\theta)\right| ^{2}-\left( \na u(R\theta)\theta\right) ^{2}\right].
\end{eqnarray*}

From $\int_{B_R} |\na u|^2\le \int_{B_R}|\na w|^2$ we infer 
\begin{eqnarray*}
\int_{B_R}|\na u|^2\le \frac{R}{n+2(\alpha-1)} \int_{\p B_R} 
\left[ \alpha^2 \dfrac {u^{2}}{R^{2}}+\left| \na u \right| ^{2}-\left( \na u\cdot \nu \right) ^{2}\right], 
\end{eqnarray*}
or equivalently 
\begin{eqnarray*}
 \int_{\p B_R} 
\left[ -\alpha^2 \dfrac {u^{2}}{R^{2}}+\left( \na u\cdot \nu \right) ^{2}\right]
\le
 \int_{\p B_R} 
\left| \na u \right| ^{2}-
\frac{n+2(\alpha-1)}R\int_{B_R}|\na u|^2\\
=
R^{n+2(\alpha-1)}\frac d{dR} \left(\frac1{R^{n+2(\alpha-1)}}\int_{B_R}|\na u|^2\right).
\end{eqnarray*}

Consequently 
\begin{eqnarray*}
\frac d{dR} \left(\frac1{R^{n+2(\alpha-1)}}\int_{B_R}|\na u|^2\right)
\ge 
\int_{\S} \frac1{R^{1+2(\alpha-1)}}(\p_\nu u)^2-\alpha^2\frac {u^2}{R^{3+2(\alpha-1)}}\\
=
\int_{\S}\frac1{R^{1+2(\alpha-1)}}\left(\p_\nu-\alpha \frac u R\right)^2-
2\alpha^2\frac{u^2}{R^{3+2(\alpha-1)}}+
2\alpha\frac u R \p_\nu u\frac1{R^{1+2(\alpha-1)}}\\
=
\int_{\S}\frac1{R^{1+2(\alpha-1)}}\left(\p_\nu u-\alpha \frac u R\right)^2
+
\alpha\frac d{dR}\int_{\S}\left(\frac{u}{R^\alpha}\right)^2.
\end{eqnarray*}

Thus if we denote %Denoting 
\begin{eqnarray*}
A(R)=\frac1{R^{n+2(\alpha-1)}}\int_{B_R}|\na u|^2-\alpha \int_{\S}\left(\frac{u}{R^\alpha}\right)^2
\end{eqnarray*}
we obtain 
\begin{eqnarray*}
A'(R, u)\ge \int_{\S}\frac1{R^{1+2(\alpha-1)}}\left(\p_\nu u-\alpha \frac u R\right)^2\ge 0.
\end{eqnarray*}

\end{proof}

Now we can finish the proof of Theorem \ref{thm-1}.
\begin{proof}
By Remark \ref{rem-1} we have that 
$\{u_m\}_{m=1}^\infty$ is bounded in $C^{1, \beta}$ for a subsequence.
Using the scaling properties of $A(R, u)$ we see that for $s>0$ we have 
\[
A(s R_{i_m}, u)=A(s, u_m).
\]
Consequently, if $0<s_1<s_2$ then by Lemma \ref{lem-1} we have 
\[
A(s_2R_{i_m}, u)-A(s_1R_{i_m}, u)
=
A(s_2, u_m)-A(s_1, u_m)\ge \int_{s_1}^{s_2}\left[
 \int_{\S}\frac1{R^{1+2(\alpha-1)}}\left(\p_\nu u_m-\alpha \frac {u_m} R\right)^2\right]dR\ge 0.
\] 
Using a customary compactness argument we see that there is a function $u_0$ such that 
$u_m\to u_0$ in $C^{1, \beta}_{loc}(\R^n)$.
One the other hand $\lim_{m\to \infty}(A(s_2R_{i_m}, u)-A(s_1R_{i_m}, u))=0.$
\end{proof}

%\bibliography{references}

\begin{bibdiv}
\begin{biblist}

\bib{MR454350}{article}{
   author={Caffarelli, Luis A.},
   title={The regularity of free boundaries in higher dimensions},
   journal={Acta Math.},
   volume={139},
   date={1977},
   number={3-4},
   pages={155--184},
   issn={0001-5962},
   review={\MR{454350}},
   doi={10.1007/BF02392236},
}


\bib{MR567780}{article}{
   author={Caffarelli, Luis A.},
   title={Compactness methods in free boundary problems},
   journal={Comm. Partial Differential Equations},
   volume={5},
   date={1980},
   number={4},
   pages={427--448},
   issn={0360-5302},
   review={\MR{567780}},
   doi={10.1080/0360530800882144},
}


\bib{MR3648978}{article}{
   author={Caffarelli, Luis},
   author={Ros-Oton, Xavier},
   author={Serra, Joaquim},
   title={Obstacle problems for integro-differential operators: regularity
   of solutions and free boundaries},
   journal={Invent. Math.},
   volume={208},
   date={2017},
   number={3},
   pages={1155--1211},
   issn={0020-9910},
   review={\MR{3648978}},
   doi={10.1007/s00222-016-0703-3},
}

\bib{MR3904453}{article}{
   author={Figalli, Alessio},
   author={Serra, Joaquim},
   title={On the fine structure of the free boundary for the classical
   obstacle problem},
   journal={Invent. Math.},
   volume={215},
   date={2019},
   number={1},
   pages={311--366},
   issn={0020-9910},
   review={\MR{3904453}},
   doi={10.1007/s00222-018-0827-8},
}

\bib{Friedman}{book}{
   author={Friedman, Avner},
   title={Variational principles and free-boundary problems},
   series={Pure and Applied Mathematics},
   note={A Wiley-Interscience Publication},
   publisher={John Wiley \& Sons, Inc., New York},
   date={1982},
   pages={ix+710},
   isbn={0-471-86849-3},
   review={\MR{679313}},
}


\bib{MR3855748}{article}{
   author={Ros-Oton, Xavier},
   title={Obstacle problems and free boundaries: an overview},
   journal={SeMA J.},
   volume={75},
   date={2018},
   number={3},
   pages={399--419},
   issn={2254-3902},
   review={\MR{3855748}},
   doi={10.1007/s40324-017-0140-2},
}
	
\bib{MR1714335}{article}{
   author={Weiss, Georg S.},
   title={A homogeneity improvement approach to the obstacle problem},
   journal={Invent. Math.},
   volume={138},
   date={1999},
   number={1},
   pages={23--50},
   issn={0020-9910},
   review={\MR{1714335}},
   doi={10.1007/s002220050340},
}

\end{biblist}
\end{bibdiv}

\end{document}